\documentclass{amsart}
\usepackage[utf8]{inputenc}
\usepackage{xypic,eucal,enumerate}

\usepackage{tikz-cd}
\usepackage{amssymb}
\usepackage{mathtools}

\usepackage[a4paper,margin=1.2in]{geometry}

\title[Canonical liftings of Calabi--Yau hypersurfaces]{Canonical liftings of Calabi--Yau hypersurfaces:\\
Dwork hypersurfaces}
\author{Przemysław Grabowski}
\date{\today}

\newcommand{\cO}{\mathcal{O}}

\newcommand{\op}{\operatorname}
\renewcommand{\phi}{\varphi}

  \newtheorem{thm}{Theorem}[section]
  \newtheorem{lemma}[thm]{Lemma}
  \newtheorem{prop}[thm]{Proposition}
  \newtheorem{cor}[thm]{Corollary}

  \theoremstyle{definition} 
  \newtheorem{defin}[thm]{Definition}
  \newtheorem{remark}[thm]{Remark}
  \newtheorem{example}[thm]{Example}

\begin{document}

\begin{abstract}
We explicitly compute canonical liftings modulo $p^2$ in a sense of Achinger--Zdanowicz of Dwork hypersurfaces. The computation involves studying a compatibility between Hodge filtrations and a crystalline Frobenius. In particular, remarkably, we explicitly compute a partial data of the crystalline Frobenius modulo $p^2$.
\end{abstract}

\maketitle

\section*{Introduction}

For every variety over a field of positive characteristic  one can consider a space of all its liftings to characteristic zero. This space can be fully described in terms of deformation theory. Moreover, some special families of varieties such as abelian varieties admit distinguished elements in their spaces of liftings: canonical liftings. These can be used to form a functor from abelian varieties in positive characteristic to abelian varieties in characteristic zero. This functor, the canonical lifting, is a fascinating object to study with many applications. Nevertheless, abelian varieties are not the only family admitting canonical liftings, another one consists of Calabi--Yau varieties. In this paper, we compute some explicit equations related to canonical liftings of Calabi--Yau projective hypersurfaces.

Let $X_0$ be a smooth and proper scheme of dimension $d$ over a perfect field $k$ of characteristic $p>0$ such that $\omega_{X_0}\simeq\cO_{X_0}$. In \cite{AchingerZdanowicz}, generalizing \cite{KatzSerreTate}, \cite{DeligneSerreTate}, \cite{Nygaard}, it was observed that if $X_0$ is \emph{$1$-ordinary} meaning that the Frobenius $F^*\colon H^d(X_0, \cO_{X_0})\to H^d(X_0, \cO_{X_0})$ is bijective then $X_0$ admits a \emph{canonical flat lifting} $X_{\rm can}$ over $W_2(k)$, described as a closed subscheme of $W_2(X)$. For $p>2$ and under certain technical conditions (satisfied in most cases of practical interest), $X_{\rm can}$ can be characterized as the unique lifting for which the crystalline Frobenius morphism
\[ 
    \phi \colon H^d_{\rm dR}(X/W_2(k)) \to H^d_{\rm dR}(X/W_2(k)) 
\]
preserves the Hodge filtration ${\rm Fil}^\bullet$, Theorem 5.7.1 in \cite{AchingerZdanowicz}. However, neither the construction nor the above characterization allow for a reasonably explicit description of $X_{\rm can}$, say in terms of the equations defining $X_0$. In fact, already for elliptic curves a ``computational'' approach to the canonical lifting presents serious difficulties (see \cite{SatohCanonical}, \cite{satoh2003fast}). 

The goal of this paper is the infinitesimal study of the interplay between the crystalline Frobenius and the Hodge filtration, with the above problem in mind. Our methods are limited to working modulo $p^2$ for $p>2$, and give a completely explicit answer only for hypersurfaces. 
In this case (and again under technical conditions), by Mazur's divisibility estimates \cite{Mazur},
$\phi\colon H^d_{\rm dR}(X/W_2(k))\to H^d_{\rm dR}(X/W_2(k))$ vanishes on ${\rm Fil}^2$ and is divisible by $p$ on ${\rm Fil}^1$. It follows that the obstruction to the compatibility $\phi({\rm Fil}^i)\subseteq {\rm Fil}^i$ is a map of vector spaces over $k$
\[ 
    \gamma\colon H^{*-1}(X_0, \Omega^1_{X_0}) = {\rm Fil}^1/{\rm Fil}^2 H^*(X_0/k) \to {\rm Fil}^1/{\rm Fil}^0 H^*(X_0/k) = H^*(X_0, \mathcal{O}_{X_0}).
\]
suitably induced by $\phi/p$. Our main technical result, Corollary \ref{C1.4} and Proposition \ref{frobenius-final}, is the computation of this map. The main surprise is that in the case of hypersurfaces, we can express $\gamma$ in terms of the cohomology of the structure sheaf of the non-reduced hypersurface defined by the equation $f^2=0$. We wonder whether cohomology of the structure sheaf of larger nilpotent thickenings of $X$ can be used to study the crystalline Frobenius modulo $p^n$ for $n>2$.

In general, our method produces a system of linear equations one needs to solve to compute the canonical lifting. In the special case of Dwork hypersurfaces
\[ 
    X_0(\lambda) \quad \colon \quad \lambda \left( x_0^{N+1} +\cdots +x_{N}^{N+1}\right) = (N+1) x_0\cdot \ldots \cdot x_{N} \quad \subseteq \quad \mathbb{P}^{N}_{k}
\]
we can be more explicit, giving a formula for the canonical lifting (Theorem \ref{main theorem}). It is again a Dwork hypersurface $X(\eta) \subset \mathbb{P}^{N}_{W_2(k)}$ whose parameter $\eta\in W_2(k)$ is the unique one satisfying the following equation
\[
(N+1)\eta^{p}\mathbb{H}\mathbb{D}_{N+1}^{p-1}(\eta)
+\phi(\eta)\mathbb{H}\mathbb{D}_{N+1}^{2p-1}(\eta)=0,
\]
where $\phi$ is the canonical Frobenius lifting of $W_2(k)$. Moreover, the polynomials $\mathbb{H}\mathbb{D}_{N+1}^{mp-1}$ for $m=1,2$ admit the following congruences modulo $p^2$: 
\[
\mathbb{H}\mathbb{D}_{N+1}^{mp-1}(X)
\equiv{\left(-(N+1)\right)}^{mp-1}\sum_{i=0}^{\left[\frac{mp-1}{N+1}\right]}
\frac{\left(1-mpH_{i(N+1)}\right)\prod_{j=1}^{N}\{\frac{j}{N+1}\}_i}{{(i!)}^{N}} X^{i(N+1)} \ \left(mod \ p^2\right),
\]
where $H_k=\sum_{i=1}^k \frac{1}{i}$ and $\{a\}_i\coloneqq a(a+1)(a+2)\ldots(a+i-1)$, Proposition \ref{extended katz formulas}.
Interestingly, the above formula features truncated hypergeometric functions. It would be interesting to give this computation a more intrinsic meaning. 

\subsection*{Acknowledgements} 

Some of the results of this paper were a part of the author's master's thesis (University of Warsaw 2020).
The author would like to thank his master's supervisor Piotr Achinger: Piotr, you are awesome. And, we would like to thank Frits Beukers for sharing with us his unpublished computations of canonical liftings performed by a different method. Our results agree with his 
formulas.

The author was supported by NCN SONATA grant number 2017/26/D/ST1/00913.

\subsection*{Notation}

If $M$ is a flat $\mathbf{Z}/p^2$-module and $M_0 = M/pM$, we denote by $\times p\colon M_0\to M$ the unique map whose precomposition with the projection $M\to M_0$ is $p\colon M\to M$.

\section{Partial computation of the crystalline Frobenius}

\subsection{}
Let $X$ be a smooth scheme over $W_2(k)$. We write $X_0 = X\otimes_{W_2(k)} k$ and denote 
by $H = H^n_{\rm dR}(X/W_2(k))$ a fixed de Rham cohomology group, 
by ${\rm Fil}^i H$ its Hodge filtration, 
and by $\phi\colon H\to H$ the crystalline Frobenius. We assume that the Hodge cohomology groups $H^j(X, \Omega^i_{X/W_2(k)})$ are free and that the Hodge spectral sequence
\[ 
    E_1^{ij} = H^j(X, \Omega^i_{X/W_2(k)} )
    \quad \Rightarrow \quad
    H^{i+j}_{\rm dR}(X/W_2(k))
\]
degenerates, so that we have 
\[
    {\rm gr}^i H := {\rm Fil}^i H/{\rm Fil}^{i+1} H \simeq H^{n-i}(X, \Omega^i_{X/W_2(k)}).
\]

\subsection{}
Since $p>2$, Mazur's divisibility estimates \cite{Mazur} imply that
\[
    \phi({\rm Fil}^1 H)\subseteq pH\quad \text{and}\quad\phi({\rm Fil}^i H) = 0 \quad\text{for $i\ge 2$}, 
\]
and therefore $\phi$ induces a map $\op{gr} \phi\colon \op{gr}^1 H \to \op{gr}^0 H$ which vanishes modulo $p$.
(Indeed, we have 
\[
\op{Fil}^2 H\subset \op{Fil}^1 H \to pH \subset H \to \op{gr}^0 H
\]
and $\op{Fil}^2 H$
is mapped to zero.)
Since $\op{gr}^0 H$ is flat over $W_2(k)$, there exists a unique map
\[
    \gamma = \text{``$\phi/p$''}\colon H^{n-1}(X_0, \Omega^1_{X_0/k}) \to H^n(X_0, \mathcal{O}_{X_0}),
\]
making the following square commute
\[ 
    \xymatrix{
        \op{gr}^1 H \ar[d] \ar[r]^{\op{gr}\phi} & \op{gr}^0 H  \\
        \op{gr}^1 H_0 \ar[r]_\gamma  & \op{gr}^0 H_0. \ar[u]_{\times p}
    }
\]
Consequently, the following conditions are equivalent: 
\begin{itemize}
    \item $\phi({\rm Fil}^i H)\subseteq {\rm Fil}^i H$ for all $i\geq 0$,
    \item $\phi({\rm Fil}^1 H)\subseteq {\rm Fil}^1 H$,
    \item $\gamma=0$.
\end{itemize}
The goal of this section is to explicate the map $\gamma$.

\subsection{}
Let us first recall the construction of the crystalline Frobenius $\phi$ (see \cite[0 3.2.3]{Illusie_deRhamWitt}, and \cite{BerthelotOgus} for a comprehensive treatment). Let $i\colon X\to P$ be a closed immersion into a smooth $W_2(k)$-scheme $P$ defined by an ideal $I\subseteq \mathcal{O}_P$, and suppose that $P$ is endowed with a lifting $F\colon P\to P$ of the absolute Frobenius. We let $\overline{P}$ denote the PD-envelope of $i$ (with respect to the standard divided power structure on $(W_2(k), (p))$), so that we have a factorization
\[
    \xymatrix{
        X \ar@/^1.2em/[rr]^i \ar[r]_{\overline{i}} & \overline{P}\ar[r] & P.
    }
\]
The natural map from the PD-envelope of $X_0$ in $P$ to $\overline{P}$ is an isomorphism of schemes \cite[3.20 1)]{BerthelotOgus}, and consequently $F$ induces a map $\overline{F}\colon \overline{P}\to \overline{P}$. 
Indeed, the universal property of a PD-envelope \cite[3.19]{BerthelotOgus} says that there is a unique map making the following diagram commutative.
\begin{center}
    \begin{tikzcd}
    X_0\arrow[r]\arrow[d,"F"]&\overline{P}\arrow[d,dashed, "\overline{F}"]\arrow[r]&P\arrow[d, "F"]\\
    X_0\arrow[r]&\overline{P}\arrow[r]&P
    \end{tikzcd}
\end{center}


If $\Omega^\bullet_{\overline{P}} = \Omega^\bullet_{\overline{P}/W_2(k), PD}$ denotes the PD de Rham complex [Ill 0 3.1], and $\op{Fil}^i_{\overline{I}} \Omega^\bullet_{\overline{P}} $ denotes the subcomplex 
\[ 
    \overline{I}^{[i]} \to \overline{I}^{[i-1]}\Omega^1_{\overline{P}/W_2(k), PD} \to \cdots,
\]
the map $i$ induces quasi-isomorphisms $\Omega^\bullet_{\overline{P}}\simeq \Omega^\bullet_{X/W_2(k)}$ and $\op{Fil}^i_{\overline{I}} \Omega^\bullet_{\overline{P}}\simeq \Omega^{\bullet\geq i}_{X/W_2(k)}$ (see \cite[Appendix 2]{Mazur}),
and consequently an isomorphism of filtered $W_2(k)$-modules
\[ 
    (H=H^n(X/W_2(k)), \op{Fil}^\bullet) \simeq (H^n(\overline{P}, \Omega^\bullet_{\overline{P}}), \op{Fil}^\bullet_{\overline{I}}).
\]
The map $\overline{F}$ induces a morphism $\phi_F\colon \Omega^\bullet_{\overline{P}}\to \Omega^\bullet_{\overline{P}}$, and the induced endomorphism of 
\[
H=H^n(X/W_2(k))\simeq H^n(\overline{P}, \Omega^\bullet_{\overline{P}})
\]
is the crystalline Frobenius $\phi$.

\subsection{}
We shall now explicate the complex $K \coloneqq \Omega^\bullet_{\overline{P}}/\op{Fil}^2_{\overline{I}}\Omega^\bullet_{\overline{P}}$.

\begin{lemma}\label{L1.1}
  Let $A$ be a ring in which $p^n A = 0$ and let $F\colon A\to A$ be a lifting of Frobenius. Let $I\subseteq A$ be an ideal such that $F(I) \subseteq I^p + p^m A$ for some $m\leq n$ (which is automatic for $m=1$). Then 
  \[ F(I^a) \subseteq I^a \quad\text{for every integer} \quad a\geq \frac{p}{p-1}\left(\left\lceil \frac{n}{m} \right\rceil - 1\right). \]
  In particular, if $p^2 A = 0$, then for every ideal $I\subseteq A$ we have $F(I^2)\subseteq I^2$.
\end{lemma}

\begin{proof}
Let $j=\lceil \frac n m \rceil$ be the smallest integer such that $mj\geq n$. We have
\[ 
  F(I^a) \subseteq (I^p + p^m A)^a = I^{pa} + p^m I^{p(a-1)} + \cdots + p^{m(j-1)} I^{p(a-j+1)}, 
\]
and all summands are contained in $I^a$ since $p(a-j+1)\geq a$.

The last part follows from the general statement by putting $n=2, m=1, a=2$.
\end{proof}

\begin{lemma}\label{L1.2}
Let $Y\subseteq P$ be the closed subscheme defined by the ideal $I^2$.
\begin{enumerate}[(a)]
    \item We have an isomorphism $\cO_{\overline{P}}/\overline{I}^{[2]}\simeq \cO_P/I^2 = \cO_Y$.
    \item The map $F\colon P\to P$ preserves the ideal $I^2$ and hence induces a lifting of Frobenius on $Y$. 
\end{enumerate}
\end{lemma}

\begin{proof}
     We claim that $\cO_{\overline{P}}/\overline{I}^{[2]}\simeq \cO_P/I^2$. It follows from the construction of $\overline{P}$ from $P$ and $I$ made in the proof of \cite[3.19]{BerthelotOgus}. Indeed, elements of $\cO_{\overline{P}}$ are equivalence classes of elements from $\cO_P$ with some extra free variables depending on $I$. Then, dividing by $\overline{I}^{[2]}$ does two things. First, it kills all variables that are not identified with any elements from $\cO_P$. Second, it kills elements from $I^2$. This gives the isomorphism.
\end{proof}

Consider the two-term complex
\[
    K = \Omega^\bullet_{\overline{P}} / \op{Fil}^2_{\overline{I}}\Omega^\bullet_{\overline{P}} = \left[\cO_Y \xrightarrow{d} \Omega^1_{P/W_2(k)}|_X \right],
\]
with the differential induced by $d\colon \mathcal{O}_P\to \Omega^1_{P/W_2(k)}$. 
We equip $K$ with the two-step filtration induced by $\op{Fil}^\bullet_{\overline{I}} $, explicitly: 
\[ 
    {\rm Fil}^2 K = 0 \quad \subseteq \quad  {\rm Fil}^1 K = \left[I/I^2 \xrightarrow{d} \Omega^1_{P/W_2(k)}|_X \right] 
    \quad \subseteq \quad {\rm Fil}^0 K = K,
\]
so that we have quasi-isomorphisms ${\rm gr}^0 K \simeq \mathcal{O}_X$ and ${\rm gr}^1 K \simeq \Omega^1_{X/W_2(k)}[-1]$, the latter induced by the conormal sequence of $i\colon X\to P$. 
A priori, the complex $K$ is only on $\overline{P}$, however, we clearly see it is a pushforward of a complex from $Y$ whose gradations are pushforwards of sheaves from $X$.
Moreover, the map $\phi_F\colon \Omega^\bullet_{\overline{P}} \to \Omega^\bullet_{\overline{P}}$ preserves $\op{Fil}^2_{\overline{I}}$ and hence induces a map $\phi\colon K\to K$. 
Indeed, the preservation follows from easy calculations similar to the ones from the proof of Lemma \ref{L1.1}; we take a typical element from a set $\overline{I}^{[2-i]}\Omega^i_{\overline{P}/W_2(k)}$ and we apply $\overline{F}$ to it, then after few elementary manipulations the result is obviously in $\op{Fil}^2_I$ especially by $p^2=0$, $\overline{F}(\Omega^i_{\overline{P}/W_2(k)})\subset p^i \Omega^i_{\overline{P}/W_2(k)}$, $\overline{F}\gamma_i=\gamma_i F$, and the fact that for any $x$ there is an $x'$ such that $F(x)=x^p+px'$.

\begin{thm}\label{Th1.3}
    The quasi-isomorphisms $\op{Fil}^i_{\overline{I}} \Omega^\bullet_{\overline{P}}\simeq \Omega^{\bullet\geq i}_{X/W_2(k)}$ induce an isomorphism $H^n(X, K)\simeq H/{\rm Fil}^2 H$ compatible with the maps $\phi$ and identifying the image of $H^n(X, {\rm Fil}^1 K) \simeq H^{n-1}(X, \Omega^1_X)$ with $\op{Fil}^1 H/\op{Fil}^2 H$. 
\end{thm}

\begin{proof}
It follows from a general fact that for any short exact sequence of chain complexes $0\to B\to A\to B/A\to0$ such that for every $n$ the natural map $H^n(B)\to H^n(A)$ is an inclusion we have a short exact sequence $0\to H^n(B)\to H^n(A)\to H^n(B/A)\to 0$. Indeed, this fact follows from an observation that the maps $H^n(B)\to H^n(A)$ being inclusions is equivalent to connecting maps $H^n(A/B)\to H^{n+1}(B)$ being zero in the long exact sequence.

The isomorphisms in the theorem follows from $B=\op{Fil}^2_{\overline{I}} \Omega^\bullet_{\overline{P}}, A= \Omega^\bullet_{\overline{P}}$ and $B=\op{Fil}^2_{\overline{I}} \Omega^\bullet_{\overline{P}}, A=\op{Fil}^1_{\overline{I}} \Omega^\bullet_{\overline{P}}$. In both of the cases, the assumption about inclusions follows from the degeneration of the Hodge spectral sequence via the quasi-isomorphisms.
\end{proof}

\subsection{} Here we conclude how to check if $\gamma=0$.

\begin{cor}\label{C1.4}
The map $\gamma$ is the unique map making the diagram below commute.
\[ 
    \xymatrix{
        H^{n-1}(X_0, \Omega^1_{X_0/k})  \ar[rrr]^\gamma & & & H^n(X_0, \mathcal{O}_{X_0})\ar[d]_{\times p} \\
        H^{n-1}(X, \Omega^1_{X/W_2(k)}) \ar[r]\ar[u] \ar[d]_{\text{conormal seq.}} & H^n(Y, K) \ar[d] \ar[r]_\phi & H^n(Y, K) \ar[d] \ar[r] & H^n(X, \mathcal{O}_X)  \ar@{=}[d]  \\
        H^{n}(X, I/I^2) \ar[r]\ar[d]_d & H^n(Y, \mathcal{O}_Y) \ar[r]_{F^*} &  H^n(Y, \mathcal{O}_Y) \ar[r] & H^n(X, \mathcal{O}_X)\\
        H^n(X,\Omega_P |_X)
    }
\]
\end{cor}

\begin{proof}
By Theorem \ref{Th1.3} and the definitions above, the top rectangle is
\begin{center}
    \begin{tikzcd}
    \op{gr}^1 H_0 \ar[rrr, "\gamma"] & & & \op{gr}^0 H_0  \ar[d, "\times p"'] \\
    \op{gr}^1 H \ar[r] \ar[u]  & H/\op{Fil}^2 H  \ar[r, "\phi"] & H/\op{Fil}^2 H   \ar[r] & \op{gr}^0 H   
    \end{tikzcd}
\end{center}
what is the definition of $\gamma$. The map $H^n(Y, K)\to H^n(Y, \mathcal{O}_Y)$ comes from a short exact sequence
$0\to \Omega^1_{P}|_X[-1]\to K \to \cO_Y\to 0$. The rest follows from the fact that $\phi$ is induced by $F$ and $F$ restricts to $Y$ by Lemma \ref{L1.2}.
\end{proof}

Finally, we can state an explicit criterion for the map $\gamma$ to vanish.

\begin{cor}\label{obstruction}
    The map $\gamma$ is zero if and only if the map $H^n(X,I/I^2)\to H^n(X,\cO_X)$ (from the diagram from Corollary \ref{C1.4}) restricted to the kernel of $d$ is zero.
\end{cor}

\begin{proof}
    The map $\gamma$ is zero if and only if the map $H^{n-1}(X, \Omega^1_{X/W_2(k)})\to H^n(X, \mathcal{O}_X)$ is zero. The later factorizes through $H^n(X,I/I^2)$. But, by exactness, the image in $H^n(X,I/I^2)$ is the kernel of $d$. So, we have $H^{n-1}(X, \Omega^1_{X/W_2(k)})\to \op{ker}(d)\to H^n(X, \mathcal{O}_X)$. The first arrow is surjective, therefore the composition is zero if and only if $\op{ker}(d)\to H^n(X, \mathcal{O}_X)$ is zero.
\end{proof}

Often, the group $H^{n-1}(X, \Omega^1_{X/W_2(k)})$ is the kernel of $d$, for example for $P=\mathbb{P}^{n+1}_{W_2(k)}$, $X=V(f)$ a projective hypersurface, and $n\ge 3$. However, this observation is not relevant for our main result. Anyway, it is important to point out that the most bottom horizontal arrows provide an explicit way how to compute crystalline Frobenius modulo $p^2$. This is remarkable.

\section{Hypersurfaces}\label{section2}

The goal of this section is to describe the map $H^n(X,I/I^2)\to H^n(X,\cO_X)$ from Corollary \ref{obstruction} in terms of homogeneous polynomials for a projective hypersurface $X$ of dimension $n$ and degree $d$. 

\subsection{} Let $A$ be a ring, let $N$ be a natural number $\ge 2$, let $S = A[x_0, \ldots, x_N]$ be the standard-graded polynomial ring, and let $P = \op{Proj} S = \mathbb{P}^N_A$. Recall that the graded $S$-module $S^\vee := \bigoplus_{t\in \mathbf{Z}} H^N(P, \cO_P(t))$ is naturally isomorphic to the quotient of the Laurent polynomial ring $A[x_0^{\pm 1},\ldots, x_N^{\pm 1}]$ by the $S$-submodule generated by monomials $x_0^{a_0}\cdots x_N^{a_N}$ where $a_i\geq 0$ for at least one $i$, see \cite[Chapter III, 5]{Hartshorne}. 
As such, $S^\vee_d = H^N(P, \cO_P(d))$ has a basis consisting of monomials in $x_0^{a_0}\cdots x_N^{a_N}$ with $a_i<0$ and $\sum a_i = d$. Identifying $S^\vee_{-N-1}$ with $A$ using the basis $x_0^{-1}\cdots x_N^{-1}$, we thus have perfect pairings $S_d \times S^\vee_{-N-1-d}\to A$ which are the Serre duality pairings $H^0(P, \cO_P(d))\times H^N(P, \omega_P(-d))\to H^n(P, \omega_P)\simeq A$ using the isomorphism $\omega_P\simeq \cO_P(-N-1)$ defined by the element $\frac{dx_0}{x_0}\cdots \frac{dx_N}{x_N} \in H^0(P, \omega_P(N+1))$. Finally, we have the Euler sequence
\[ 
  0\to \Omega^1_P\to \cO_P(-1)^{N+1} \xrightarrow{(x_0, \ldots, x_N)} \cO_P\to 0
\]
where the first map sends $\omega$ to $(\omega\cdot \frac{d}{dx_i})$.

\subsection{}\label{cohomology of OX} Let $f\in S_d$ ($d\geq 1$) be a homogeneous element which is a nonzerodivisor, let $R = S/(f)$, and let $X = \op{Proj} R = V(f)\subseteq P$ be the hypersurface cut out by $f$. We set $n=N-1 \ge 1$. The ideal sheaf $I\subseteq \cO_P$ of $X$ is identified with $\cO_P(-d)$. This gives the short exact sequence $0\to \cO_P(-d)\to \cO_P\to\cO_X\to 0$ whose long exact sequence can be used to compute an identification of the graded $R$-module $R^\vee := \bigoplus_{t\in \mathbb{Z}} H^n(X, \cO_X(t))$ with the kernel of the map $f\colon S^\vee(-d)\to S^\vee$. 
In particular, we have 
$H^n(X, \cO_X)\simeq S^\vee_{-d}$, $H^n(X, \cO_X(-1))\simeq S^\vee_{-d-1}$, and 
\[
H^n(X,I/I^2)=H^n(X, \cO_X(-d))\simeq (A\cdot f)^\perp \subseteq S^\vee_{-2d}
\] 
since $\cO_X(-d)\simeq I/I^2$.

\subsection{} A standard calculation of cohomology groups via long exact sequences, see for example \cite[Pages 3, 4]{ComplexVectorBundles}, applied to twists of the Euler sequence and to a sequence $0\to \Omega^1_{P/A}(-d)\to \Omega^1_{P/A} \to \Omega^1_{P/A}|_X\to 0$ for $n\ge 2$ gives that 
\[
H^n(X,\Omega^1_{P/A}|_X)\simeq H^{n+1}(P, \Omega^1_{P/A}(-d)) \simeq \op{ker}( {S^\vee_{-d-1}}^{\oplus N+1}\xrightarrow{(x_0, \ldots, x_N)} S^\vee_{-d})
\]
and for $n=1$ we get a short exact sequence
\[
0\to H^1(P,\Omega^1_{P/A})\simeq S_0\to H^1(X,\Omega^1_{P/A}|_X) \to H^2(P, \Omega^1_{P/A}(-d)) \simeq \op{ker}( {S^\vee_{-d-1}}^{\oplus N+1}\xrightarrow{(x_0, x_1,x_2)} S^\vee_{-d})\to 0.
\]
In particular, $H^{n-1}(P,\Omega^1_{P/A}|_X)=0$ if and only if $n\ne 2$. And, for $n=2$, it is $S_0$.

\subsection{} We shall now explicate the composition of maps 
\[
H^{n}(X, I/I^2)\xrightarrow{d} H^n(X,\Omega^1_{P/A} |_X) \to H^{n+1}(P, \Omega^1_{P/A}(-d))
\]
The second arrow is always surjective, and it is an isomorphism for $n\ge 2$.

Suppose that $X$ is smooth over $A$, in which case we have the conormal sequence
\[ 
  0\to I/I^2 \xrightarrow{d} \Omega^1_{P/A}|_X\to \Omega^1_{X/A}\to 0.
\]
We can compute this $d$ at a level of cohomology groups explicitly by the following lemma.

\begin{lemma}\label{redarrow}
    The map $H^{n}(X, I/I^2) \to H^{n+1}(P, \Omega^1_{P/A}(-d))$ under the injection $H^n(X,I/I^2) \subseteq S^\vee_{-2d}$ and the isomorphism $H^{n+1}(P, \Omega^1_{P/A}(-d)) \simeq \op{ker}( {S^\vee_{-d-1}}^{\oplus N+1}\xrightarrow{(x_0, \ldots, x_N)} S^\vee_{-d})$ projected onto $i$-th summand of ${S^\vee_{-d-1}}^{\oplus N+1}$
    is given by the following formula:
    \[
    S^\vee_{-2d} \ni g \mapsto f\cdot \frac{dg}{dx_i} + 2\frac{df}{dx_i} \cdot g \in S^\vee_{-d-1}.
    \]
\end{lemma}

\begin{proof}
    We have the following diagram.
    \begin{center}
    \begin{tikzcd}
         I^2=I(-d) \ar[d, "d"]\ar[r]& I \ar[d, "d"]\ar[r]& I/I^2 \ar[d, "d"] \\
         \Omega^1_{P/A}(-d)\ar[r]& \Omega^1_{P/A} \ar[r]& \Omega^1_{P/A}|_X 
    \end{tikzcd}   
    \end{center}
In this diagram, all $d$'s are induced by a universal derivation on $P$. In particular, they are not $\cO_P$-linear, so they do not induce morphisms between cohomology groups naturally. However, we can use these maps in local computations on representatives in \v{C}ech cohomology to get our formula anyway.

A bit abusively, it is done the following way.

Let $f^2 g$ be an element in $H^n(P,I^2)$ representing an element $g$ in $S^\vee_{-2d}$. Then, by the Leibniz rule, we have that $d(f^2 g)=f^2 dg + 2f gdf =f(fdg+2gdf)$. This means that $g$ is mapped to $fdg+2gdf$. The formula is proved.
\end{proof}

In the above formula, $d/dx_i$ act as differential operators of degree $-1$ on $S$ and on $A[x_0^{\pm 1}, \ldots, x_N^{\pm 1}]$ preserving the kernel of the projection onto $S^\vee$ and hence induce maps $S^\vee \to S^\vee(-1)$.

\subsection{}
Let $Y\coloneqq V(f^2)\subseteq P$.
The same calculation from \ref{cohomology of OX} performed for $f^2$ gives us that $\bigoplus_{t\in \mathbb{Z}} H^n(Y,\cO_Y(t))=\op{ker}(f^2:S^\vee(-2d)\to S^\vee)$. Moreover, $X$ being a closed subscheme of $Y$ gives us a short exact sequence
$
0\to I/I^2 \to \cO_Y \to \cO_X \to 0
$.
This sequence tells us two things. 

First, the map
$H^{n}(X, I/I^2) \to H^n(Y, \mathcal{O}_Y)$ under the above isomorphisms is the inclusion $(A\cdot f)^\perp\subseteq S^\vee_{-2d}$. 

Second, it tells us that the map
$H^n(Y, \mathcal{O}_Y) \to H^n(X, \mathcal{O}_X)$ is a restriction of the multiplication by $f$. Indeed, we have maps between parts of long exact sequences 
\begin{center}
    \begin{tikzcd}
        H^n(X,\cO_Y)\ar[r, hook]\ar[d] & H^N(P,\cO_P(-2d))=S^\vee_{-2d}\ar[r, "f^2"]\ar[d,"f"]  & H^N(P,\cO_P)\ar[d, equal]\\
        H^n(X,\cO_X)\ar[r, hook] & H^N(P,\cO_P(-d))=S^\vee_{-d}\ar[r, "f"]  & H^N(P,\cO_P).\\
    \end{tikzcd}
\end{center}

\subsection{}
We shall now compute the endomorphism $F^*$ on $H^n(Y,\cO_Y)\simeq S^\vee_{-2d}$.

\begin{prop}\label{frobenius}
Let $k$ be a perfect field of characteristic $p>2$. Let $A=W_2(k)$ be Witt vectors of length $2$ over $k$.

Let $F$ be a Frobenius lifting on $P$. Then $F$ restricts to a Frobenius lifting on $Y$ that induces the map
\[
F^*: g\mapsto F^*(g) \times F^*(f^2)/f^2
\]
on the cohomology group $H^{n}(Y,\cO_Y)\cong S_{-2d}^\vee$.
\end{prop}

\begin{proof}
    Let $F(Y)\coloneqq V(F^*(f^2))\subseteq P$ be a hypersurface. We have the following diagram.
    \begin{center}
    \begin{tikzcd}
    0\arrow[r]& \cO_P(-2d)\arrow[d,"F^*"] \arrow[r, "f^2"]& \cO_P\arrow[d, "F^*"] \arrow[r]& \cO_Y\arrow[d, "F^*"]\arrow[r]& 0\\
    0\arrow[r]& \cO_P(-2dp)\arrow[d," F^*(f^2)/f^2"'] \arrow[r, " F^*(f^2)"]& \cO_P\arrow[d,equal] \arrow[r]& \cO_{F(Y)}\arrow[d, "mod \ f^2"]\arrow[r]& 0\\
    0\arrow[r]& \cO_P(-2d) \arrow[r, " f^2"]& \cO_P \arrow[r]& \cO_Y\arrow[r]& 0
    \end{tikzcd}.
\end{center}
This diagram makes sense, because $f^2$ divides $F^*(f^2)$. This is a part of Lemma \ref{L1.1}. Finally, the formula is obtained by using long exact sequences of cohomology groups for this simple diagram.
\end{proof}

\subsection{}
Finally, we can combine all the above into the following proposition.
\begin{prop}\label{frobenius-final}
    Let $k$ be a perfect field of characteristic $p>2$,
    let $A=W_2(k)$, and let $X\coloneqq V(f)\subseteq P=\mathbb{P}^N_{W_2(k)}$ be a smooth hypersurface of degree $d\ge 1$. Let $F$ be a Frobenius lifting on $P$, then the composition of the maps
    \[
    H^{n}(X, I/I^2) \to H^n(Y, \mathcal{O}_Y) \xrightarrow{F^*}  H^n(Y, \mathcal{O}_Y) \to H^n(X, \mathcal{O}_X)
    \]
    from Corollary \ref{C1.4} under isomorphisms
    $H^n(X,I/I^2) \simeq \op{ker}(f: S^\vee_{-2d}\to S^\vee_{-d})$ and $H^n(X,\cO_X)\simeq S^\vee_{-d}$ is given by the formula
    \[
    g\mapsto F^*(g) \times F^*(f^2)/f = F^*(g)\times  \left(-f^{2p-1}+2f^{p-1}F^*(f)\right).
    \]
\end{prop}

\begin{proof}
    Apart from the identity $F^*(f^2)/f=-f^{2p-1}+2f^{p-1}F^*(f)$, all follows from the whole discussion above. To prove the identity, observe that there is an element $\delta(f)=``\frac{\Phi^*(f)-f^p}{p} "$ such that $F^*(f)=f^p+p\delta(f)$. Therefore, we see
    \begin{align*}
    F^*(f^2)/f={\left(f^p+p\delta(f)\right)}^2 /f
    &=f^{2p-1}+2pf^{p-1}\delta(f)\\
    &=f^{2p-1}+2pf^{p-1}\frac{F^*(f)-f^p}{p}\\
    &=f^{2p-1}+2f^{p-1}(F^*(f)-f^p)\\
    &=f^{2p-1}-2f^{2p-1}+2f^{p-1}F^*(f)\\
    &=-f^{2p-1}+2f^{p-1}F^*(f).
    \end{align*}
\end{proof}

We finish with a simple example of a Frobenius lifting on $\mathbb{P}^N_{W_2(k)}$.

\begin{example}\label{frobenius lifting}
    Let $S = W_2(k)[x_0, \ldots, x_N]$ be the standard graded ring over $W_2(k)$. Therefore, we have $\op{Proj} S=\mathbb{P}^N_{W_2(k)}$. 

    Consider a unique ring morphism $\Phi:S \to S$ given by the canonical lifting of Frobenius on coefficients $W_2(k)$, i.e. $(a,b)\mapsto (a^p,b^p)$ in Witt coordinates, and by $x_i \mapsto x_i^p$ on variables. This $\Phi$ induces a Frobenius lifting $F$ on $\mathbb{P}^N_{W_2(k)}$.
\end{example}

\section{Dwork hypersurfaces}
The goal of this section is to compute examples of canonical liftings modulo $p^2$. We perform it for a Dwork family of hypersurfaces, because they are highly symmetrical and this can be exploited to decrease the amount of computation required significantly.

Let $k$ be a perfect field of characteristic $p>2$.  Let $W_2(k)$ be Witt vectors of length two over $k$. Let $N\ge 2$.

\begin{defin}
    Let $A$ be a ring. A \emph{Dwork hypersurface} (over $A$)  for a parameter $\lambda\in A$ is a hypersurface cut out by the equation 
\[ 
    \lambda \left( x_0^{N+1} +\cdots +x_{N}^{N+1}\right) = (N+1) x_0\cdot \ldots \cdot x_{N} 
\]
in $\mathbb{P}^{N}_{A}=\op{Proj}(A[x_0, \ldots, x_N])$, where $A[x_0, \ldots, x_N]$ is the standard graded ring over $A$.

If $A=k$, then we denote this hypersurface by $X_0(\lambda)$. If $A=W_2(k)$, then we denote this hypersurface by $X(\lambda)$.
\end{defin}

\begin{remark}
    Clearly, given $N$ fixed, Dwork hypersurfaces over $A$ form an affine line family of projective hypersurfaces. This family admits an extension to a projective line family whose point at the infinity is a Fermat hypersurface cut out by the equation $x_0^{N+1} +\cdots +x_{N}^{N+1}$. Over $A=k$, whenever this hypersurface is ordinary and smooth, then it admits a canonical lifting modulo $p^2$. It is again a Fermat hypersurface cut out by the same equation, but over $W_2(k)$. This follows from exactly the same calculations as in this section, but simpler. As such, this computation can be treated as an exercise.
\end{remark}

The computation is performed in the following way. We begin with an ordinary, smooth Dwork hypersurface $X_0=X_0(\lambda)$. We take its flat lifting $X=X(\eta)$ over $W_2(k)$. (In particular, $\eta \equiv\lambda \ (\text{mod} \ p)$.) The Hodge spectral sequence of $X$ degenerates by \cite[SGA7 Vol. II Expose XI Theoreme 1.5]{sga7}. Thus, by Section \ref{section2}, we know that $X$ is the canonical lifting modulo $p^2$ of $X_0$ if and only if the crystalline Frobenius is compatible with the Hodge filtration of $X$. This is equivalent to the map $\gamma$ being zero. And this, by Corollary \ref{obstruction}, Lemma \ref{redarrow}, and Proposition \ref{frobenius}, happens if and only if an image of an explicit arrow between explicit modules of homogeneous polynomials is zero. However, all of the maps involved are equivariant under linear symmetries of $X$. (These are introduced below.) 
Consequently, after a small argument, this arrow is zero if and only if it is zero on invariant elements of its source. The module of invariant elements is cyclic, and we will provide its explicit generator. Therefore, it is enough to compute for each $\eta$ the image of the generator is zero. This will turn out to be controlled by a single linear equation.

So, the computation is straightforward, however it still requires preparations.

\subsection{Symmetries}\label{symmetries}
Fix $N\ge 2$. Consider the following two groups of isomorphisms of $S_1$, where $P=\mathbb{P}^{N}_{A}=\op{Proj}(S=A[x_0, \ldots, x_N])$:
\begin{itemize}
    \item $G_1 \quad : \quad$ permutations of variables $x_i$,
    \item $G_2 \quad : \quad$ maps given by $x_i\mapsto \lambda_i x_i$, where $\lambda_i\in A$, $\lambda_i^{N+1}=1$, and $\lambda_0 \cdot \lambda_1 \cdot \ldots \cdot \lambda_N=1$.
\end{itemize}
Each of these groups induces a group of isomorphisms of $P$. Clearly, all of them preserve all Dwork hypersurfaces. Moreover, these groups intersect trivially, and $G_1$ naturally acts on $G_2$. Therefore, the group generated by both of them, we call it $G$, is their semiproduct. In particular, it is a finite group. This leads to the following lemma.

\begin{lemma}\label{order of G}
    Let $A$ be $k$ or $W_2(k)$.
    If $p > N+1$, then the order of $G$ is not divisible by $p$.
\end{lemma}

\begin{proof}
    We assume that $A$ is $k$ or $W_2(k)$ only to control a group of roots of unity. It is not an essential assumption.

    Let $p>N+1$, then the group of $(N+1)$-roots of unity is a subgroup of a cyclic group of order $N+1$. Then, the group $G_2$ is a subgroup of a group generated by maps given by $x_i\mapsto \lambda_i x_i$, where $\lambda_i\in A$ and $\lambda_i^{N+1}=1$. This group is of order dividing ${(N+1)}^{N+1}$. Therefore, the same is true for $G_2$. At the same time, the order of $G_1$ is $(N+1)!$. Consequently, as $G$ is a semiproduct of $G_1$ and $G_2$, we have that the order of $G$ divides $(N+1)!{(N+1)}^{N+1}$ and this number is not divisible by $p$.
\end{proof}

We will need a partial information about a ring of invariant elements under $G$ action.

\begin{lemma}\label{G-invariants}
Assume that $A$ admits a primitive $N+1$-root of unity $\zeta$.
Then every $G$-invariant homogeneous polynomial in $S_{N+1}$ is of the form
$$a\sum_{i=0}^N x_i^{N+1}+b\prod_{i=0}^N x_i,$$
for some $a,b\in A$. 
\end{lemma}

\begin{proof}
    Of course, any such polynomial is $G$-invariant.
    
    Let $P$ be an element of $S_{N+1}^G$.
    $G_1$ is a subgroup of $G$, thus $P$ is $G_1$-invariant. This means that $P$ is a linear combination of symmetric polynomials. They are in bijection with monomials $m(e)\coloneqq x_0^{e_1}x_1^{e_2}\ldots x_N^{e_N}$, where $e_1\ge e_2\ge\ldots \ge e_N \ge 0$ and $\sum_{i=0}^{N} e_i=N+1$ by the operation $m(e)\mapsto \sum'_{g\in G_1} g(m(e))$, where the prime means that we sum each element of the orbit once. In particular, we have $m(N+1,0,\ldots,0) \mapsto \sum_{i=0}^N x_i^{N+1}$ and $m(1,1,\ldots,1)\mapsto \prod_{i=0}^N x_i$. Hereafter, we also use $m(e)$ to denote the polynomial it correspond to. So, we have $P=\sum_e a(e) m(e)$, where $a(e)\in A$. We are going to show that $a(e)=0$ unless $e$ is $(N+1,0,\ldots,0)$, or $(1,1,\ldots,1)$.

    First, observe that if a sum $\sum_e a(e) m(e)$ is $G_2$-invariant, then we must have for each $e$ that $g(m(e))=m(e)$ for each $g\in G_2$, or $a(e)=0$. Indeed, it follows from a simple observation that both polynomials $g(m(e))$ and $m(e)$ are spanned on the same monomials.

    Let $e\ne (N+1,0,\ldots,0), (1,1,\ldots,1)$. Let $0<k<N$ be any index such that $e_k\ne 0$. Then $0<e_k<N+1$.
    Consider $g\in G_2$ that is determined by $x_k \mapsto \zeta x_k$ and $x_{N}\mapsto \zeta^{-1}x_{N}$. Then, if we have an equality of polynomials $g(m(e))=m(e)$, then the coefficient next to the monomial $m(e)$ is $\zeta^{e_k}$, but it should be one. So, $\zeta^{e_k}=1$, but $0<e_k<N+1$. This is a contradiction with the assumption that $\zeta$ is primitive.
\end{proof}

\subsection{Smoothness} By a Nakayama lemma, e.g. \cite[Proposition B.3]{Gortz--Wedhorn}, a Dwork hypersurface $X(\eta)$ is smooth over $W_2(k)$ if and only if $X_0(\eta \ (\op{mod} \ p))$ is smooth.

\begin{lemma}\label{smooth}
    Let $X_0(\lambda)$ be a Dwork hypersurface. The hypersurface $X_0(\lambda)$ is smooth if and only if
    $\lambda^{N+1} \ne 1$, $\lambda \ne 0$, and $p$ does not divide $N+1$.
\end{lemma}

\begin{proof}
    It is a simple application of the Jacobian criterion, see e.g. \cite[I, Theorem 5.1]{Hartshorne}.
\end{proof}

\begin{remark}
    The Fermat hypersurface is smooth if and only if $p$ does not divide $N+1$.
\end{remark}

\subsection{Hasse--Dwork Polynomials} This section introduces a family of polynomials that play an essential role in computing ordinarity of Dwork hypersurfaces, what is analogical to computing a Hasse invariant, see \cite[IV, Proposition 4.21]{Hartshorne}, and, in this paper, computing canonical liftings modulo $p^2$. Therefore, we call them Hasse--Dwork polynomials.

\begin{defin}\label{hasse--dwork}
    We define a Hasse--Dwork polynomial for natural numbers $P,M$ to be the following polynomial:
    
 \[
 \mathbb{H}\mathbb{D}_M^{P}(X)={(-M)}^{P}\left(1+M!\sum_{i=1}^{\left[\frac{P}{M}\right] } {{P}\choose{M \times i}}  {\left(\frac{X}{-M}\right)}^{iM}\right),
 \]
 where $M!{{P}\choose{M \times i}}=
{{P}\choose{i}}{{P-i}\choose{i}}{{P-2i}\choose{i}}\ldots{{P-(M-1)i}\choose{i}}$, i.e. the number ${{P}\choose{M \times i}}$ is the number of all possible unordered
 collections of $M$ disjoint subsets of the set $\{1,\ldots, P\}$, each of the cardinality $i$.
\end{defin}

The following lemma explains their role.

\begin{lemma}\label{role of hasse-dwork}
An evaluation of a Hasse-Dwork polynomial $\mathbb{H}\mathbb{D}_{N+1}^P(\lambda)$ is the coefficient next to monomial
$\prod_{i=0}^N x_i^{P}$ in the polynomial
${\left(\lambda\left(\sum_{i=0}^N x_i^{N+1}\right)-\left(N+1\right)\prod_{i=0}^N x_i\right)}^{P}$.
\end{lemma}

\begin{proof}
    One can compute the coefficient in the following way.
    Let $i=0,\ldots, \left[\frac{P}{N+1}\right]$.
    Choose $(N+1)\cdot i$ factors out of $P$ that we have.
    Then, decompose it into $i$ sets of size $N+1$ each. Denote them by $A_0,\ldots,A_N$. Take the summand $\lambda x_i^{N+1}$ from the factors from $A_i$. From the factors outside the one chosen initially we take the summand $-\left(n+1\right)\prod_{i=0}^N x_i$. The product of all taken summands contributes to the coefficient, and the coefficient is the sum of all these contributions under all the choices. Namely:

    \begin{align*}
    \left({(-(N+1))}^{P}+\sum_{i=1}^{\left[\frac{P}{N+1}\right] } {{P}\choose{(N+1) \times i}}(N+1)!
    {\left(\lambda\right)}^{i(N+1)}
    {(-(N+1))}^{P-i(N+1)} \right)=\\
    ={(-(N+1))}^{P}\left(1+(N+1)!\sum_{i=1}^{\left[\frac{P}{N+1}\right] } {{P}\choose{(N+1) \times i}}  {\left(\frac{\lambda}{-(N+1)}\right)}^{i(N+1)}\right)=\mathbb{H}\mathbb{D}_{N+1}^{P}(\lambda). \quad \qedhere
\end{align*}
\end{proof}

Consequently, Hasse--Dwork polynomials have integral coefficients, so we can consider their reductions modulo $p$ and $p^2$. We recall a Katz's result \cite[2.3.7.18]{KatzAlgSoln} about a modulo $p$ case first. (Observe that that paper's notations and ours differ with respect where we put $\lambda$ in the defining equation of a Dwork hypersurface. This is why his formula is not the same as the one below, but they are equivalent.)

 \begin{prop}[Katz's formula]
 Let $p$ not divide $N+1$. Then, we have a congruence:
 $$
 \mathbb{H}\mathbb{D}_{N+1}^{p-1}(X)\equiv\sum_{i=0}^{\left[\frac{p-1}{N+1}\right]}
\frac{\prod_{j=1}^{N}\{\frac{j}{N+1}\}_i}{{(i!)}^{N}} X^{i(N+1)} \ \left(mod \ p\right),
$$
where $\{a\}_i\coloneqq a(a+1)(a+2)\ldots(a+i-1)$ for $i\ge 1$, and $\{a\}_0\coloneqq 1$.
 \end{prop}

We have the following extension of this formula modulo $p^2$.

\begin{prop}\label{extended katz formulas}
 Let $p$ not divide $N+1$. 
 Let $m=1,2$. We have congruences:
$$
\mathbb{H}\mathbb{D}_{N+1}^{mp-1}(X)
\equiv{\left(-(N+1)\right)}^{mp-1}\sum_{i=0}^{\left[\frac{mp-1}{N+1}\right]}
\frac{\left(1-mpH_{i(N+1)}\right)\prod_{j=1}^{N}\{\frac{j}{N+1}\}_i}{{(i!)}^{N}} X^{i(N+1)} \ \left(mod \ p^2\right),
$$
where $\{a\}_i\coloneqq a(a+1)(a+2)\ldots(a+i-1)$ for $i\ge 1$, $\{a\}_0\coloneqq 1$ and $H_k\coloneqq 1+\frac{1}{2}+\frac{1}{3}+\ldots+\frac{1}{k}$ is the harmonic sequence. For $i(N+1)\ge p$, we have $pH_{i(n+1)}=1+p(H_{i(n+1)}-\frac{1}{p})$.
\end{prop}

\begin{proof}
    We have to show that
    $$(N+1)!{{mp-1}\choose{(N+1) \times i}}  {\left(\frac{1}{-(N+1)}\right)}^{i(N+1)}
    \equiv 
    \frac{
\left(1-mpH_{i(N+1)}\right)\prod_{j=1}^{N}\{\frac{j}{N+1}\}_i}{{(i!)}^{N}} \ \left(mod \ p^2\right).$$

By elementary transformations, this is equivalent to
\[
{{mp-1}\choose{i(N+1)}}\equiv \left(1-mpH_{i(N+1)}\right){(-1)}^{i(N+1)} \ \left(mod \ p^2\right).
\]
This is true due to simple elementary calculations, some in a form of an induction, and an observation that $\prod_{j=1}^{N}\{\frac{j}{N+1}\}_i$ is a multiple of $p$ for $i(N+1)\ge p$.
\end{proof}

This extension can be used to get a curious result.

\begin{cor}
    Let $\eta\in W_2(k)$. Let $\lambda\coloneqq \eta \ (\op{mod} \ p )$. Let $p$ not divide $N+1$. If $\mathbb{H}\mathbb{D}_{N+1}^{p-1}(\lambda)\ne 0$, then the fraction $\frac{\mathbb{H}\mathbb{D}_{N+1}^{p-1}(\eta)}{\mathbb{H}\mathbb{D}_{N+1}^{2p-1}(\eta)}$ is well defined and it depends only on $\lambda$.
\end{cor}

\begin{proof}
    By Proposition \ref{extended katz formulas}, we can write $\mathbb{H}\mathbb{D}_{N+1}^{mp-1}(\eta)=A_m+pB_m$, where
    $\eta = \lambda+p\lambda_1$,
    $A_m=\sum_{i=0}^{\left[\frac{mp-1}{N+1}\right]}C_i^m \lambda^{i(N+1)}$,
    and
    $B_m=\sum_{i=0}^{\left[\frac{p-1}{N+1}\right]}C_i^m
i(N+1)\lambda^{i(N+1)-1}\lambda_1$, where \\
$C_i^m\coloneqq {\left(-(N+1)\right)}^{mp-1}\frac{\left(1-mpH_{i(N+1)}\right)\prod_{j=1}^{N}\{\frac{j}{N+1}\}_i}{{(i!)}^{N}}$. Moreover, we observe that $A_1 \equiv A_2$ and $B_1\equiv B_2$ modulo $p$, therefore $A_2$ is not zero modulo $p$ by the assumption $\mathbb{H}\mathbb{D}_{N+1}^{p-1}(\lambda)\ne 0$. This proves that the fraction is well defined.

The sole dependence on $\lambda$ follows from the following calculation.

$$\frac{\mathbb{H}\mathbb{D}_{N+1}^{p-1}(\eta)}{\mathbb{H}\mathbb{D}_{N+1}^{2p-1}(\eta)}=\frac{A_1+pB_1}{A_2+pB_2}=
\frac{(A_1+pB_1)(A_2-pB_2)}{A_2^2-{(pB_2)}^2}=
\frac{A_1 A_2+pA_2 B_1-pA_1 B_2}{A_2^2}
=\frac{A_1 A_2}{A_2^2}.
$$
\end{proof}

\subsection{Ordinarity}

The following is a part of \cite[2.3.7.18]{KatzAlgSoln}.

\begin{prop}\label{ordinary}
    Let $X_0(\lambda)$ be a smooth Dwork hypersurface. It is ($1$-)ordinary if and only if $\mathbb{H}\mathbb{D}_{N+1}^{p-1}(\lambda)\ne 0$.
\end{prop}

\begin{proof}
    In this case, by a computation similar the one from \cite[IV Proposition 4.21]{Hartshorne}, being ordinary is equivalent to the coefficient next to the monomial
$\prod_{i=0}^N x_i^{p-1}$ in the polynomial
${\left(\lambda\left(\sum_{i=0}^N x_i^{N+1}\right)-\left(N+1\right)\prod_{i=0}^N x_i\right)}^{p-1}$ being non-zero. By Lemma \ref{role of hasse-dwork}, this coefficient is $\mathbb{H}\mathbb{D}_{N+1}^{p-1}(\lambda)$.
\end{proof}

\subsection{Computation on Invariant Elements}

Let $X=X(\eta)\subset \mathbb{P}^{N}_{W_2(k)}$ be a Dwork hypersurface over $W_2(k)$ whose reduction modulo $p$ is smooth and ordinary. Let $f$ be the homogeneous polynomial defining $X$. The group $G$ acts on $\mathbb{P}^{N}_{W_2(k)}$ and it preserves $X$, therefore the map
$H^{n}(X, I/I^2) \to H^n(X, \mathcal{O}_X)$ from Corollary \ref{frobenius-final} is $G$-equivariant.
In this subsection, we compute the $G$-invariant image of this map.

We have isomorphisms $H^n(X,I/I^2) \simeq \op{ker}(f: S^\vee_{-2(N+1)}\to S^\vee_{-(N+1)})$ and $H^n(X,\cO_X)\simeq S^\vee_{-(N+1)}$ since the degree $d$ of $f$ is $N+1$.

\begin{lemma}\label{trival target}
    The induced action of $G$ on $H^n(X,\cO_X)$ is trivial. In particular, we have $H^n(X,\cO_X)=H^n(X,\cO_X)^G$.
\end{lemma}
\begin{proof}
    The module $H^n(X,\cO_X)\simeq S^\vee_{-(N+1)}$
    is spanned by $\prod_{i=0}^N x_i^{-1}$. This element is $G$-invariant.
\end{proof}

\begin{lemma}\label{generator}
    We have the following equality of submodules of $H^n(X,I/I^2)$:
    \[
    \op{ker}(d)^G = H^n(X,I/I^2)^G,
    \]
    where $d$ is the map from Lemma \ref{redarrow}.
    Moreover, this module is spanned by
    \[
    \left(\sum_{i=0}^N x_i^{N+1}+\eta\prod_{i=0}^N x_i\right)^\vee \coloneqq \left(\sum_{i=0}^N x_i^{-(N+1)}+\eta\prod_{i=0}^N x_i^{-1}\right)\cdot \prod_{i=0}^N x_i^{-1}.
    \]
\end{lemma}

\begin{proof}
    Of course, we have $\op{ker}(d)^G \subset H^n(X,I/I^2)^G$, so it is enough to compute the bigger one, and check if it is killed by $d$.

    Let $g^\vee \in S^\vee_{-2(N+1)}$ be a dual polynomial to $g\in S^\vee_{N+1}$, i.e. $g^\vee(x_i)\coloneqq g(x_i^{-1})\prod_{i=0}^N x_i^{-1}$. It is clear that $g^\vee$ is $G$-invariant if and only if $g$ is, and, by Lemma \ref{G-invariants}, this means that $g=a\sum_{i=0}^N x_i^{N+1}+b\prod_{i=0}^N x_i,$
    for some $a,b\in W_2(k)$.

    We compute
    \begin{align*}
        g^\vee \times f=& \left(a\sum_{i=0}^N x_i^{N+1}+b\prod_{i=0}^N x_i\right)^\vee \times \left( \eta\sum_{i=0}^N x_i^{N+1}-(N+1)\prod_{i=0}^N x_i\right)=\\
        =&{\left((N+1)a\eta-b(N+1) \right)}^{\vee}\\
        =&(N+1){\left(a\eta-b\right)}^{\vee}.
    \end{align*}
    By, Lemma \ref{smooth}, we know that $N+1$ is invertible. So, $g^\vee\times f=0$ if and only if $a\eta-b=0$. Therefore, $H^n(X,I/I^2)^G$ is spanned by $g^\vee$ for
    $g=\sum_{i=0}^N x_i^{N+1}+\eta\prod_{i=0}^N x_i$.

    We do the check if
    $2\frac{\partial f}{\partial x_j} \times g^{\vee}+ f\times \frac{\partial g^\vee}{\partial x_j}=0$. Indeed, we have
    \begin{align*}
        2\left((N+1)(\eta x_j^N-\prod_{i\ne j} x_i) \right)\times g^\vee+f\times \left( -(N+2)x_j^{N+2}-\sum_{k\ne j}x_k^{N+1}x_j-2\eta x_j\prod_{i=0}^N x_i\right)^\vee=&\\
        \left(2(N+1)\eta x_j^\vee - 2(N+1)\eta x_j^\vee\right)+\left(
   -(N+2)\eta x_j^\vee -N\eta x_j^\vee+2(N+1)\eta x_j^\vee\right)=&0.
    \end{align*}
    Therefore, $d(g^\vee)=0$ and the lemma is proved.
\end{proof}

\begin{lemma}\label{G-invariant computation}
    The $G$-invariant map induced by the map from Corollary \ref{frobenius-final}
    \[
    \op{ker}(d)^G \to H^n(X, \mathcal{O}_X)^G
    \]
    is zero if and only if
    \[
    (N+1)\eta^{p}\mathbb{H}\mathbb{D}_{N+1}^{p-1}(\eta)
+\phi(\eta)\mathbb{H}\mathbb{D}_{N+1}^{2p-1}(\eta)=0,
\]
where $\phi$ is the canonical Frobenius lifting on  $W_2(k)$.
\end{lemma}

\begin{proof}
    By Lemmas \ref{trival target} and \ref{generator}, we know that the map $\op{ker}(d)^G \to H^n(X, \mathcal{O}_X)^G$ is a map
    $S^\vee_{-2(N+1)}\supset H^n(X,I/I^2)^G=\op{span}(g^\vee) \to H^n(X, \mathcal{O}_X)=S^\vee_{-(N+1)}$.
    Therefore, it is enough to compute when the image of $g^\vee$ is zero.

    By Proposition \ref{frobenius-final}, the image is the element
    \[
     F^*(g^\vee)\times  \left(-f^{2p-1}+2f^{p-1}F^*(f) \right),
    \]
    where $F$ is a Frobenius lifting on $\mathbb{P}^N_{W_2(k)}$. Hereafter, we choose $F$ from Example \ref{frobenius lifting}.

    We have:
    \begin{align*}
    g^\vee =& \left(\sum_{i=0}^N x_i^{-(N+1)}+\eta\prod_{i=0}^N x_i^{-1}\right)\cdot \prod_{i=0}^N x_i^{-1}, \\
    f =&  \eta\sum_{i=0}^N x_i^{N+1}-(N+1)\prod_{i=0}^N x_i ,\\
    F^*(g^\vee)=&\left(\sum_{i=0}^N x_i^{-p(N+1)}+\phi(\eta)\prod_{i=0}^N x_i^{-p}\right)\cdot \prod_{i=0}^N x_i^{-p},\\
    F^*(f)=&\phi(\eta)\sum_{i=0}^N x_i^{p(N+1)}-(N+1)\prod_{i=0}^N x_i^p,\\
    \end{align*}
    where $\phi(N+1)=N+1$, because $\phi$ is the identity on integers.

    First, we compute $F^*(g^\vee) \times -f^{2p-1}$. We start by computing
    \[
    x_j^{-p(n+1)} \prod_{i=0}^N x_i^{-p} \times \left( \eta\sum_{i=0}^N x_i^{N+1}-(N+1)\prod_{i=0}^N x_i\right)^{2p-1}={{2p-1}\choose{p}}\eta^{p}\mathbb{H}\mathbb{D}_{N+1}^{p-1}(\eta)
    \]
    that follows from a simple combinatorics and Lemma \ref{role of hasse-dwork}. Next, we compute
    \[
    \prod_{i=0}^N x_i^{-2p}  \times \left( \eta\sum_{i=0}^N x_i^{N+1}-(N+1)\prod_{i=0}^N x_i\right)^{2p-1} = \mathbb{H}\mathbb{D}_{N+1}^{2p-1}(\eta) 
    \]
    using Lemma \ref{role of hasse-dwork}. These, plus an elementary observation that ${{2p-1}\choose{p}}$ is $1$ modulo $p^2$, give
    \[
    F^*(g^\vee) \times -f^{2p-1} = -(N+1)\eta^{p}\mathbb{H}\mathbb{D}_{N+1}^{p-1}(\eta) -\phi(\eta)\mathbb{H}\mathbb{D}_{N+1}^{2p-1}(\eta).
    \]

    Second, we compute $F^*(g^\vee) \times 2f^{p-1}F^*(f)$ similarly. We have
    \[
    2f^{p-1}F^*(f)=2 \left(\eta\sum_{i=0}^N x_i^{N+1}-(N+1)\prod_{i=0}^N x_i\right)^{p-1} \left(\phi(\eta)\sum_{i=0}^N x_i^{p(N+1)}-(N+1)\prod_{i=0}^N x_i^p\right).
    \]
    Therefore, we compute
    \[
    x_j^{-p(n+1)} \prod_{i=0}^N x_i^{-p} \times 2f^{p-1}F^*(f)= 2\mathbb{H}\mathbb{D}_{N+1}^{p-1}(\eta)\phi(\eta)
    \]
    and
    \[
    \prod_{i=0}^N x_i^{-2p} \times 2f^{p-1}F^*(f)= 2\mathbb{H}\mathbb{D}_{N+1}^{p-1}(\eta)(-1)(N+1).
    \]
    Consequently, we get
    \[
    F^*(g^\vee) \times 2f^{p-1}F^*(f) = (N+1)2\mathbb{H}\mathbb{D}_{n+1}^{p-1}(\eta)\phi(\eta)-\phi(\eta) 2\mathbb{H}\mathbb{D}_{n+1}^{p-1}(\eta)(N+1)=0
    \]
    We combine the above to get
   \[
     F^*(g^\vee)\times  \left(-f^{2p-1}+2f^{p-1}F^*(f) \right)=-(N+1)\eta^{p}\mathbb{H}\mathbb{D}_{N+1}^{p-1}(\eta) -\phi(\eta)\mathbb{H}\mathbb{D}_{N+1}^{2p-1}(\eta).
    \]
    This finishes the proof.
\end{proof}

\subsection{Canonical Liftings Modulo $p^2$}

In this final subsection, we show that for Dwork hypersurfaces a computation of the map from Corollary \ref{obstruction} being zero can be reduced to a computation of this map on $G$-invariants being zero. In ``most'' of the cases this follows from the following lemma.

\begin{lemma}\label{reduction}
    Let $A$ be a ring. Let $G$ be a finite group whose order is invertible in $A$. Let $f: M\to N$ be a $G$-equivariant map of $A$-modules such that the action on $N$ is trivial. Then the map $f$ is zero if and only if $f^G: M^G\to N^G$ is zero.
\end{lemma}

\begin{proof}
    The first implication is trivial.

    The second follows from the following argument. Let $m\in M$, then $f(\sum_{g \in G}gm)=\sum_{g \in G}f(gm)=|G|f(m)$, therefore $f(\sum_{g \in G}gm)=0$ if and only if $f(m)=0$. But the element $\sum_{g \in G}gm$ is in $M^G$, thus this shows that if $f^G$ is zero, then $f$ is zero.
\end{proof}

Finally, our main result is the following theorem. (For the definition of Hasse--Dwork polynomials $\mathbb{H}\mathbb{D}^{P}_{M}$ see Definition \ref{hasse--dwork}.)

\begin{thm}\label{main theorem}
    Let $k$ be a perfect field of characteristic $p>2$.
    Let $W_2(k)$ be Witt vectors of length $2$ over $k$.
    Let $N\ge 2$ be an integer. Assume $p$ does not divide $N+1$.
    Let $\lambda \in k$. Assume $\lambda \ne 0$, $\lambda^{N+1}\ne 1$, and $\mathbb{H}\mathbb{D}^{p-1}_{N+1}(\lambda)\ne 0$. Then a Dwork hypersurface $X_0(\lambda)$ defined by a homogeneous polynomial
    \[
     \lambda \left( x_0^{N+1} +\cdots +x_{N}^{N+1}\right) - (N+1) x_0\cdot \ldots \cdot x_{N} 
    \]
    in $\mathbb{P}^N_{k}=\op{Proj}(k[x_0,\ldots,x_N])$ is a smooth, and $1$-ordinary projective variety such that $\omega_{X_0(\lambda)}\simeq\cO_{X_0(\lambda)}$.
    Therefore, it admits a canonical lifting modulo $p^2$ in the sense of
    \cite{AchingerZdanowicz}.
    
    This canonical lifting is isomorphic to a Dwork hypersurface $X(\eta)$ over $W_2(k)$ defined by a homogeneous polynomial
    \[
     \eta \left( x_0^{N+1} +\cdots +x_{N}^{N+1}\right) - (N+1) x_0\cdot \ldots \cdot x_{N} 
    \]
    in $\mathbb{P}^N_{W_2(k)}$, where $\eta\in W_2(k)$, and $\lambda \equiv \eta \ (\op{mod} \ p)$, such that $\eta$ satisfies an equation
    \[
    (N+1)\eta^{p}\mathbb{H}\mathbb{D}_{N+1}^{p-1}(\eta)
+\phi(\eta)\mathbb{H}\mathbb{D}_{N+1}^{2p-1}(\eta)=0,
    \]
    where $\phi$ is the canonical Frobenius lifting on  $W_2(k)$.
\end{thm}

\begin{proof}
    Hypotheses to admit a canonical lifting from the paper \cite{AchingerZdanowicz} are satisfied by Lemma \ref{smooth}, Proposition \ref{ordinary}, and a standard calculation of a canonical divisor, e.g. \cite[II Example 8.20.3]{Hartshorne}.

    We proof that the canonical lifting is given by the equation from the statement. Let $\eta\in W_2(k)$ be such that $\lambda \equiv \eta \ (\op{mod} \ p)$. Let $X=X(\eta)$, let $X_0= X_0(\lambda)$. By Corollary \ref{obstruction}, $X$ is the canonical lifting of $X_0$ if and only if the map
    \[
    \op{ker}(d) \to H^n(X,\cO_X)
    \]
    is zero. However, what will be proved in a moment, this is equivalent to
    $\op{ker}(d)^G \to H^n(X,\cO_X)^G$ being zero. This condition, by Lemma \ref{G-invariant computation}, is precisely the equation in the theorem.

    Now, we shall proof the equivalence.

    First, we do ``most'' of the cases easily. Let $p>N+1$. Then, by Lemma \ref{order of G}, we can apply Lemma \ref{reduction} to the map $\op{ker}(d) \to H^n(X,\cO_X)$. This finishes the proof.

    However, for $p\le N+1$, the prime $p$ divides the order of $G$, so we cannot use such simple argument. Therefore, we provide the following more involved argument that works for all admissible pairs $p,N$.

    First, observe that the equation $(N+1)\eta^{p}\mathbb{H}\mathbb{D}_{N+1}^{p-1}(\eta)
+\phi(\eta)\mathbb{H}\mathbb{D}_{N+1}^{2p-1}(\eta)=0$ has at most one solution, since it is ``essentially'' linear. Therefore, the map
$\op{ker}(d)^G \to H^n(X,\cO_X)^G$ can be zero for at most one Dwork hypersurface. So, if we know that a canonical lifting of a Dwork hypersurface is a Dwork hypersurface, then we know that the map $\op{ker}(d) \to H^n(X,\cO_X)$ is zero for only one $\eta$. This would finish the proof in full generality. In the rest of this proof, we prove that a canonical lifting of a Dwork hypersurface is a Dwork hypersurface.

Indeed, recall a construction of a canonical lifting
\cite[Section 4.1.]{AchingerZdanowicz}. Let $\sigma$ be a $G$-invariant Frobenius splitting of $\mathbb{P}^{N}_{k}$ that descends to a Frobenius splitting on $X_0$. Such splitting exists by \cite[Exercise 1.3.E (5)]{BrionKumar}. Let $I_\sigma$ be an ideal sheaf from the construction for $\mathbb{P}^N_k$, and let $J_\sigma$ be such ideal for $X_0$.
Let $X$ be a canonical lifting of $X_0$, by the construction $V(J_\sigma)\simeq X$.
Then, we have the following diagram.  
\begin{center}
    \begin{tikzcd}
        V(I_\sigma) \ar[d,equal] \ar[r, hook]& W_2(\mathbb{P}^{N}_{k})\ar[d]& \ar[l]W_2(X_0)& V(J_\sigma)\ar[l,hook']\ar[d,equal]\\
        V(I_\sigma) \ar[r,"\simeq"]& \mathbb{P}^N_{W_2}\ar[d, "\op{mod}  p"]& X\ar[d, "\op{mod}  p"]\ar[u,hook]\ar[l,hook']&V(J_\sigma)\ar[l,"\simeq"']\\
        &\mathbb{P}^{N}_{k}&X_0\ar[l,hook']&
    \end{tikzcd}
\end{center}
In the above diagram, the map $W_2(\mathbb{P}^{N}_{k})\to \mathbb{P}^N_{W_2}$ is given by a Teichm{\" u}ller lifting of the line bundle $\mathcal{O}_{\mathbb{P}^N_{k}}(1)$. Moreover, the action of $G$ lifts to a $G$-action on $W_2(\mathbb{P}^{N}_{k})$ by using Teichm{\" u}ller liftings, and this action descends to the standard $G$-action on $\mathbb{P}^N_{W_2}$: permutations act by permutations, and a multiplication by a scalar acts by multiplying by the Teichm{\" u}ller lifting of the scalar. Now, since $X$ is flat over $W_2(k)$, of codimension $1$ in $\mathbb{P}^{N}_{W_2(k)}$, and $G$-invariant, we conclude that $X$ is a $G$-invariant hypersurface. So, by Lemma \ref{G-invariants}, it is a Dwork hypersurface.
This finishes the proof that a canonical lifting of a Dwork hypersurface is a Dwork hypersurface. The theorem is proved.
\end{proof}

\begin{remark}
In the above proof, we proved that a canonical lifting of a Dwork hypersurface is a Dwork hypersurface.
However, this fact is not completely obvious. A priori, it could happen that the canonical lifting is not a Dwork hypersurface. Then, there would exist exactly one Dwork hypersurface faking being a canonical lifting by having the map
$\op{ker}(d) \to H^n(X,\cO_X)$ not equal zero, but having the map $\op{ker}(d)^G \to H^n(X,\cO_X)^G$ equal zero. By our argument, this cannot happen.
\end{remark}

\bibliographystyle{amsalpha} 
\bibliography{bib}

\providecommand{\bysame}{\leavevmode\hbox to3em{\hrulefill}\thinspace}
\providecommand{\MR}{\relax\ifhmode\unskip\space\fi MR }
\providecommand{\MRhref}[2]{%
  \href{http://www.ams.org/mathscinet-getitem?mr=#1}{#2}
}
\providecommand{\href}[2]{#2}
\begin{thebibliography}{OSS80}

\bibitem[AZ21]{AchingerZdanowicz}
Piotr Achinger and Maciej Zdanowicz, \emph{Serre-{T}ate theory for
  {C}alabi-{Y}au varieties}, J. Reine Angew. Math. \textbf{780} (2021),
  139--196. \MR{4333980}

\bibitem[BK05]{BrionKumar}
Michel Brion and Shrawan Kumar, \emph{Frobenius splitting methods in geometry
  and representation theory}, Progress in Mathematics, vol. 231, Birkh\"auser
  Boston, Inc., Boston, MA, 2005.
  \MR{\href{http://www.ams.org/mathscinet-getitem?mr=2107324}{2107324}}

\bibitem[BO78]{BerthelotOgus}
Pierre Berthelot and Arthur Ogus, \emph{Notes on crystalline cohomology},
  Princeton University Press, Princeton, N.J.; University of Tokyo Press,
  Tokyo, 1978.
  \MR{\href{http://www.ams.org/mathscinet-getitem?mr=0491705}{0491705}}

\bibitem[Del81]{DeligneSerreTate}
P.~Deligne, \emph{Cristaux ordinaires et coordonn\'{e}es canoniques}, Algebraic
  surfaces ({O}rsay, 1976--78), Lecture Notes in Math, vol. 868, Springer,
  Berlin-New York, 1981, With the collaboration of L. Illusie., With an
  appendix by Nicholas M. Katz., pp.~pp 80--137. \MR{638599}

\bibitem[GW10]{Gortz--Wedhorn}
Ulrich G\"{o}rtz and Torsten Wedhorn, \emph{Algebraic geometry {I}}, Advanced
  Lectures in Mathematics, Vieweg + Teubner, Wiesbaden, 2010, Schemes with
  examples and exercises. \MR{2675155}

\bibitem[Har77]{Hartshorne}
Robin Hartshorne, \emph{Algebraic geometry}, Springer-Verlag, New York, 1977,
  Graduate Texts in Mathematics, No. 52.
  \MR{\href{http://www.ams.org/mathscinet-getitem?mr=0463157}{0463157} (57
  \#3116)}

\bibitem[Ill79]{Illusie_deRhamWitt}
Luc Illusie, \emph{Complexe de de\thinspace {R}ham-{W}itt et cohomologie
  cristalline}, Ann. Sci. \'Ecole Norm. Sup. (4) \textbf{12} (1979), no.~4,
  501--661. \MR{\href{http://www.ams.org/mathscinet-getitem?mr=565469}{565469}}

\bibitem[Kat72]{KatzAlgSoln}
Nicholas~M. Katz, \emph{Algebraic solutions of differential equations
  ({$p$}-curvature and the {H}odge filtration)}, Invent. Math. \textbf{18}
  (1972), 1--118.
  \MR{\href{http://www.ams.org/mathscinet-getitem?mr=0337959}{0337959}}

\bibitem[Kat81]{KatzSerreTate}
\bysame, \emph{Serre-{T}ate local moduli}, Algebraic surfaces ({O}rsay,
  1976--78), Lecture Notes in Math., vol. 868, Springer, Berlin-New York, 1981,
  pp.~138--202.
  \MR{\href{http://www.ams.org/mathscinet-getitem?mr=638600}{638600}}

\bibitem[Maz73]{Mazur}
Barry Mazur, \emph{Frobenius and the {H}odge filtration (estimates)}, Ann. of
  Math. (2) \textbf{98} (1973), 58--95.
  \MR{\href{http://www.ams.org/mathscinet-getitem?mr=0321932}{0321932}}

\bibitem[Nyg83]{Nygaard}
Niels~O. Nygaard, \emph{The {T}ate conjecture for ordinary {$K3$}\ surfaces
  over finite fields}, Invent. Math. \textbf{74} (1983), no.~2, 213--237.
  \MR{\href{http://www.ams.org/mathscinet-getitem?mr=723215}{723215}}

\bibitem[OSS80]{ComplexVectorBundles}
Christian Okonek, Michael Schneider, and Heinz Spindler, \emph{Vector bundles
  on complex projective spaces}, Progress in Mathematics, vol.~3,
  Birkh\"{a}user, Boston, MA, 1980. \MR{561910}

\bibitem[Sat00]{SatohCanonical}
Takakazu Satoh, \emph{The canonical lift of an ordinary elliptic curve over a
  finite field and its point counting}, J. Ramanujan Math. Soc. \textbf{15}
  (2000), no.~4, 247--270. \MR{1801221}

\bibitem[sga73]{sga7}
\emph{Groupes de monodromie en g\'{e}om\'{e}trie alg\'{e}brique. {II}}, Lecture
  Notes in Mathematics, vol. Vol. 340, Springer-Verlag, Berlin-New York, 1973,
  S\'{e}minaire de G\'{e}om\'{e}trie Alg\'{e}brique du Bois-Marie 1967--1969
  (SGA 7 II), Dirig\'{e} par P. Deligne et N. Katz. \MR{354657}

\bibitem[SST03]{satoh2003fast}
Takakazu Satoh, Berit Skjernaa, and Yuichiro Taguchi, \emph{Fast computation of
  canonical lifts of elliptic curves and its application to point counting},
  Finite Fields and Their Applications \textbf{9} (2003), no.~1, 89--101.

\end{thebibliography}

\end{document}